\documentclass[11pt]{amsart}

\usepackage{epigamath-voisin}

\usepackage[english]{babel}

\numberwithin{equation}{section}

\usepackage[shortlabels]{enumitem}

\usepackage{mathrsfs}
\usepackage{amssymb}
\usepackage{amsmath}
\usepackage{amsfonts}
\usepackage{latexsym}
\usepackage[all,ps]{xy}
\RequirePackage{amsthm}
\RequirePackage{amssymb}
\usepackage[all,ps]{xy}
\usepackage{amscd}

\newtheorem{thm}{Theorem}[section]

\newtheorem{prop}[thm]{Proposition}
\newtheorem{lem}[thm]{Lemma}
\newtheorem{conj}[thm]{Conjecture}

\theoremstyle{definition}

\theoremstyle{remark}
\newtheorem{rmk}[thm]{Remark}

\newcommand{\OG}{\mathrm{OG}}
\newcommand{\Km}{\mathrm{Kum}}

\def\Z{{\mathbb Z}}

\def\R{{\mathbb R}}
\def\Q{{\mathbb Q}}
\def\C{{\mathbb C}}
\def\X{{\tilde X}}

\newcommand{\IP}{\mathbb{P}}
\newcommand{\IC}{\mathbb{C}}
\newcommand{\IR}{\mathbb{R}}

\newcommand{\IZ}{\mathbb{Z}}

\newcommand{\frM}{\mathfrak{M}}

\def\Aut{\mathop{\rm Aut}\nolimits}

\def\Def{\mathop{\rm Def}\nolimits}

\def\Hom{\mathop{\rm Hom}\nolimits}

\def\lra{\longrightarrow}

\def\NS{\mathop{\rm NS}\nolimits}

\def\Pic{\mathop{\rm Pic}\nolimits}

\def\tilde{\widetilde}

\def\phi{\varphi}

\def\D{{\tilde D}}

\def\t{{\tilde \tau}}

\def\Eff{\mathop{\rm Eff}\nolimits}
\def\Nef{\mathop{\rm Nef}\nolimits}

\renewcommand\int{\mathop{\rm int}}
\DeclareMathOperator{\id}{id}
\DeclareMathOperator{\diag}{diag}

\def\Eff{\mathop{\rm Eff}\nolimits}

\def\Pic{\mathop{\rm Pic}\nolimits}
\def\dim{\mathop{\rm dim}\nolimits}

\def\rk{\mathop{\rm rk}\nolimits}
\def\GL{\mathop{\rm GL}\nolimits}
\def\Mov{\mathop{\rm Mov}\nolimits}
\def\Bir{\mathop{\rm Bir}\nolimits}
\def\Jac{\mathop{\rm Jac}\nolimits}

\newcommand{\supth}[1]{\ensuremath{#1^{\mathrm{th}}}}
\newcommand{\suprd}[1]{\ensuremath{#1^{\mathrm{rd}}}}

\YearArticle{2026} \EpigaArticleNr{22} \ReceivedOn{March 21, 2023}
\InFinalFormOn{April 1, 2026}
\AcceptedOn{April 17, 2026}

\title{On the cone conjecture for Enriques manifolds}
\titlemark{On the cone conjecture for Enriques manifolds}

\author{Gianluca Pacienza}
\address{Universit\'e de Lorraine, CNRS, IECL, 
F-54000 Nancy, France}
\email{gianluca.pacienza@univ-lorraine.fr}
\author{Alessandra Sarti}
\address{Laboratoire de Mathématiques et Applications, CNRS et Université de Poitiers, 86073 Poitiers Cedex 9, France}
\email{Alessandra.Sarti@math.univ-poitiers.fr}

\authormark{G.~Pacienza and A.~Sarti}

\AbstractInEnglish{Enriques manifolds are non--simply connected manifolds whose universal cover is irreducible holomorphic symplectic, and as such they are natural generalizations of Enriques surfaces.
The goal of this note is to prove the Morrison--Kawamata cone conjecture 
for very general Enriques manifolds when the degree of the cover is prime. The proof uses the analogous result (established by Amerik--Verbitsky) for their universal cover. We also verify the conjecture for a very general Enriques manifold
which is deformation equivalent to one of the known examples. }



\acknowledgement{G.\,P.\ is partially supported by the CNRS  International  Emerging Actions project ``Birational and arithmetic aspects of orbifolds''. A.\,S.\ is partially supported by the ANR project No. ANR-20-CE40-0026-01 (SMAGP).
G.\,P.\ and A.\,S.\ are partially supported by ANR project No. ANR-23-CE40-0026 (POK0).}

\dedication{\`A Claire, avec gratitude et admiration}

\begin{document}



\maketitle

\begin{prelims}

\DisplayAbstractInEnglish

\bigskip

\DisplayKeyWords

\medskip

\DisplayMSCclass

\end{prelims}


\newpage

\setcounter{tocdepth}{1}

\tableofcontents


%
\section{Introduction}
%
An Enriques manifold is a connected complex manifold $X$ which is not simply connected and whose universal cover $\X$ is an irreducible holomorphic symplectic (IHS) manifold. Enriques manifolds were simultaneously and independently introduced in \cite{BNWS} and \cite{OS1} as a natural generalization of Enriques surfaces. The fundamental group of an Enriques manifold is a cyclic and finite group $G=\langle g\rangle$; its order is called the {\it index} of $X$; it is also the order of the torsion canonical class $K_X\in \Pic(X)$. From the definition it follows that an Enriques manifold is even dimensional. Moreover, as $X=\X/G$, it is compact and, since $h^{2,0}(X)=h^{0,2}(X)=0$ (\textit{cf.}~\textit{e.g.}~\cite[Proposition 2.6]{OS1}), it turns out that an Enriques manifold is always projective (\textit{cf.}  \cite[Proposition 2.1(4)]{BNWS} and \cite[Corollary 2.7]{OS1}) and so is its universal cover.  In those papers examples of Enriques manifolds of index 2, 3 and 4 are constructed (these are still the only known examples so far), while their periods are studied in \cite{OS2}.

The Morrison--Kawamata cone conjecture ({see \cite{Mor, Ka}) concerns the action of the automorphism group of manifolds  (and more generally mildly singular pairs) with numerically trivial canonical class on the cone of rational nef classes and predicts the existence of a rational, polyhedral fundamental domain for such action.
  
\begin{conj}[Morrison--Kawamata cone conjecture]\label{conj:cone}
Let $X$ be a  projective manifold with numerically trivial canonical class.
\begin{enumerate}[label={\rm(\arabic*)}, ref={\rm\arabic*}]
\item\label{c:cone-1} 
There is a rational polyhedral  cone $\Pi$ which is a fundamental domain for the
action of $\Aut(X)$ on  the convex hull $\Nef^+(X)$ of\, $\Nef(X)\cap N^1(X)_{\Q}$ in $N^1(X)_{\R}$; i.e. 
\begin{enumerate}[label={\rm(\alph*)}, ref={\rm\alph*}]
\item $\Nef^+(X)= \cup_{g\in \Aut(X)} g^*\Pi$, 
\item $\int (\Pi) \cap \int (g^*\Pi) = \emptyset$ unless $g^* = \id$ in $\GL(N^1(X)_\R)$.
\end{enumerate}
\item\label{c:cone-2} 
There is a rational polyhedral  cone $\Pi'$ which is a fundamental domain in the sense above for the
action of $\Bir(X)$ on the convex hull ${\Mov}^+(X)$ of\, $\overline{\Mov}(X)\cap N^1(X)_{\Q}$ inside $N^1(X)_{\R}$. 
\end{enumerate}
\end{conj}

Item~\eqref{c:cone-2} above is also known as the birational cone conjecture.  Notice that in the literature there is also a version of the conjecture where $\Nef^+(X)$ is replaced by $\Nef^e(X):=\Nef(X)\cap \Eff(X)$ (the same for ${\Mov}^+(X)$, which is replaced by $\overline{\Mov}^e(X):=\overline{\Mov}(X)\cap \Eff(X)$).  This statement is analogous to the classical one for the four known deformations types of IHS manifolds (and equivalent in general modulo the SYZ-conjecture; \textit{cf.}~\cite[Remark 1.4]{MY}).  The conjecture (we will not specify which version of it and refer the interested readers to the papers we quote) has been proved in dimension 2 by Sterk--Looijenga, Namikawa, Kawamata, and Totaro (see \cite{Ste,Na, Ka, Tot}), by Prendergast-Smith \cite{PS} for abelian varieties, by Amerik--Verbitsky \cite{AV1,AV2} for IHS manifolds, building upon the birational cone conjecture established by Markman in \cite{Mark-survey}, see also Markman--Yoshioka \cite{MY}, and by
Lehn--Mongardi--Pacienza \cite{LMP} for singular IHS varieties.  For Calabi--Yau varieties the conjecture is open in general. For recent results in this direction, see \cite{GLW22, GLSW24} and the references therein.  We refer the reader to \cite{Tot-survey,LOP18} for nice introductions to this topic. The conjecture is deeply related to birational geometry. Item~\eqref{c:cone-1} of the conjecture yields the finiteness, up to automorphisms, of birational contractions and fiber space structures of the initial variety, while item~\eqref{c:cone-2} implies, modulo standard conjectures of the MMP, the finiteness of minimal models, up to isomorphisms, of any $\Q$-factorial and terminal variety with non-negative Kodaira dimension (\textit{cf.}~\cite[Theorem 2.4]{CL14}).

By the Beauville--Bogomolov decomposition theorem \cite{Beau83}, we know that any $K$-trivial manifold $V$ admits a finite \'etale cover $\tilde V\to V$, where $\tilde V$ is a product of Calabi--Yau manifolds, IHS manifolds and an abelian variety. A general question is: suppose we know Conjecture~\ref{conj:cone} for $\tilde V$, can we deduce it for $V$?  Our main result provides a positive answer when we have only one factor of IHS type and the IHS has in some sense the smallest possible Picard group. More precisely, we show the following.

\begin{thm}\label{identity}
Let $X=\X/G$ be an Enriques manifold, where $G=\langle g\rangle$ is a finite cyclic group acting freely on an IHS manifold $\X$.  Assume that the action of\, $G$ on $\Pic(\X)$ is the identity or, equivalently, that $\Pic(\X)=H^2(\X,\IZ)^G$.  Then there is a rational and polyhedral cone which is a fundamental domain for the action of\, $\Aut(X)$ $($resp.\ of $\Bir(X))$ on $\Nef^+(X)$ $($resp.\ on $\Mov^+(X))$.
\end{thm}

For the equivalence mentioned in the hypotheses of Theorem~\ref{identity}, we refer the reader to Lemma~\ref{lem:inc}.  We put together in the following statement the cases in which the hypothesis $\Pic(\X)=H^2(\X,\IZ)^G$ can be verified, and the consequences that can be drawn. Notice that all the known examples of Enriques manifolds are covered by the statement.

\begin{thm}\label{thm:main}
Let $X=\X/G$ be an Enriques manifold, where $G=\langle g\rangle$ is a cyclic group of order $d$ acting freely on an IHS manifold $\X$. Conjecture~\ref{conj:cone}, items~\eqref{c:cone-1} and~\eqref{c:cone-2}, hold in the following cases:
\begin{enumerate}[label={\rm(\roman*)}, ref={\rm\roman*}]
\item\label{t:main-1} The  integer $d$ is prime, and the  Enriques manifold $X$ is  very general in the moduli space. 
\item\label{t:main-2} The manifold $X$ is a very general deformation of the index~$4$ examples provided in  \cite{BNWS, OS1}.  
\item\label{t:main-3} The manifold $\X$ is of\, $K3^{[n]}$ type $($resp.\ of\, $\Km_n$ type$)$, and the index $d$ is one of\, $13,17,19, 23, 46$ $($resp.\ one of\, $5,7,9,14,18)$.
\end{enumerate}
\end{thm}

In Theorem~\ref{thm:main} the conjecture follows from Theorem~\ref{identity} and the fact that in those cases we can prove that $\Pic(\X)=H^2(\X,\IZ)^G$.  

As for the proof of Theorem~\ref{identity}, if $\pi\colon \X\to X=\X/G$ is the covering map, by Amerik--Verbitsky \cite{AV1,AV2} there exists a rational polyhedral convex cone $\D$ which is a fundamental domain for the action of $\Aut(\X)$ on $\Nef^{+}(\X)$. We set $ D:=\D\cap \pi^* N^1(X)_{\mathbb R}.  $ Recall that $N^1(X)$ denotes the N\'eron--Severi group of $X$, which coincides with the Picard group of $X$ (after tensoring by $\mathbb Q$ or $\mathbb R$).  The proof of Theorem~\ref{identity} then consists in showing that $ \pi_*(D)$ yields a fundamental domain for the action of $\Aut(X)$ on $\Nef^{+}(X)$: the rationality and polyhedrality of $D$ are implied by those of $\D$.  The main point is to show that if $\xi\in \pi^* N^1(X)_{\mathbb R}$ and $\phi$ is an automorphism of $\X$ such that $\phi^*(\xi)\in \D$, then $\phi$ ``descends'' to an automorphism of $X$, \textit{i.e.}~$\phi$ commutes with $G$. To show this commutativity, we check it on cohomology; this is where we use the
assumption $\Pic(\X)=H^2(\X,\IZ)^G$.  Some care has to be taken, as it is well known that there are non-trivial automorphisms acting trivially on cohomology for the two deformation classes coming from abelian surfaces (see Remark~\ref{rmk:ker} for details).  As for the proof of Theorem~\ref{thm:main}, item~\eqref{t:main-1} is proved in Proposition~\ref{prop:gen}, after we have recalled the construction of moduli spaces of marked Enriques manifolds, while items~\eqref{t:main-2} and~\eqref{t:main-3} are the content of several propositions in Sections~\ref{ss:notprime} and~\ref{sec:4.2}.

We conclude the introduction by mentioning some developments after the appearance of our work and closely related to it. 
Firstly, motivated by the approach followed in this paper, Monti and Quedo have proved the cone conjecture for all \'etale quotients of abelian varieties; see~\cite{MQ}. More recently, Gachet \cite{G25}  proved in particular the cone conjecture for abstract Enriques manifolds, regardless of their deformation type or index. Finally, the preprint \cite{BGGG} answers negatively a question raised in this paper about the possibility of constructing Enriques manifolds
from $\OG 10$ manifolds. Hence, despite the fact that the results of the present paper are now much weaker than those obtained in  \cite{G25}, we believe that they served both as a motivation and as  evidence for the cone conjecture for Enriques manifolds and revived the interest towards this class of varieties and their singular analogues; see \textit{e.g.}~\cite{BCS24,DRTX,DTX}.

\subsection*{Acknowledgments}
We thank A.~Chiodo for his kind invitation to Jussieu in December 2021, where this project started, and L.~Kamenova, G.~Mongardi and A.~Oblomkov, the organizers of the workshop ``Hyperk\"ahler quotients, singularities and quivers'' at the Simons Center for Geometry and Physics in January 2023, where the project was continued. We thank S.~Boissi\`ere, C.~Bonnaf\'e, E.~Floris, A.~Garbagnati, E.~Markman and G.~Mongardi for very useful exchanges. Finally we warmly thank the referees whose extremely detailed reports helped us to correct several mistakes and improve the readability of the paper.

%
\section{Preliminaries}
%

\subsection{Cones}\label{subsec:cones}
Let $V$ be a real vector space of finite dimension. We suppose that $V$ has a distinguished $\mathbb Q$-structure, \textit{i.e.}~there exists a rational vector space $V_{\mathbb Q}$ such that $V= V_{\mathbb Q}\otimes_{\mathbb Q}\mathbb R$.  A cone is a subset $C\subset V$ such that for all $\lambda \in \mathbb R_{>0}$ and all $x \in C$ we have $\lambda x\in C$.  It is called polyhedral (resp.\ rational polyhedral) if it is a closed convex cone generated by a finite number of vectors (resp.\ rational vectors). It is non-degenerate if it contains no line. If $C\subset V$ is an open convex cone, we denote by $C_+$ the convex hull of $\overline {C}\cap V_{\mathbb Q}$ in $V$.  Let $\Gamma$ be a group and $\rho\colon\Gamma\to \GL(V)$ be a group homomorphism. Suppose that the action of $\Gamma$ on $V$ (via $\rho$) leaves an open convex cone $C\subset V$ invariant. In this case we say that $(C_+,V)$ is of polyhedral type if there is a polyhedral cone $\Pi\subset C_+$ such that $\Gamma\cdot \Pi
\supset C$.

\begin{lem}[\textit{cf.} {\cite[Theorem~3.8 and Application~4.14]{Loo14}}]\label{lem:loo}
Under the notation and assumptions above, suppose that $\rho\colon\Gamma\to \GL(V)$ is injective and that the cone $C$ is non-degenerate. If\, $(C_+, \Gamma)$ is of polyhedral type, $C_+$ admits a rational polyhedral fundamental domain for the action of\, $\Gamma$.
\end{lem}

\subsection{Basic facts on IHS manifolds}
An irreducible holomorphic symplectic (IHS) manifold is a compact K\"ahler manifold $\X$ which is simply connected and carries a holomorphic symplectic 2-form $\sigma$ such that $H^0(\X,\Omega^2_\X)=\C\cdot \sigma$. For a general introduction to the subject we refer to \cite{Huy-basic}.  Let $\X$ be an irreducible holomorphic symplectic manifold of dimension $2n \geq 2$. Let $\sigma\in H^0(\X,\Omega^2_\X)$ be such that $\int_{\X} \sigma^n \bar{\sigma}^{n}=1$.  Then, following \cite{Beau83}, the second cohomology group $H^2(\X,\mathbb C)$ is endowed with a quadratic form $q=q_{\X}$ which is non-degenerate and, up to a positive multiple, is induced by an integral non-divisible quadratic form on $H^2(\X,\mathbb Z)$ of signature $(3, b_{2}(\X)-3)$. The form $q$ is called the Beauville--Bogomolov--Fujiki quadratic form of $\X$.  By Fujiki \cite{Fuj} there exists a positive rational number $c=c_\X$ (the Fujiki constant of $\X$) such that $c \cdot q^{n}(\alpha) = \int_{\X}\alpha^{2n},\ \forall \alpha \in H^2(\X,\mathbb Z)$.  For an IHS manifold $\tilde X$ we have $\Pic(\tilde X)=\NS(\tilde X)$, and we will
use both notations interchangeably.

\begin{rmk}\label{rmk:ker}
  If $\X$ is an IHS manifold, the kernel of $\nu\colon\Aut(\X)\rightarrow {\rm O}(H^2(\X,\IZ))$
  is finite, by \cite[Proposition~9.1]{Huy-basic}.  Notice moreover that by a result of Hassett and Tschinkel \cite[Theorem 2.1]{HT}, the kernel of   $\nu$ is invariant under smooth deformations of the manifold  $\X$. This allows us to
 compute it for all four known deformation types of IHS manifolds, thanks to \cite[Proposition 10]{Beau83},  \cite[Corollary 3.3]{BNWS} and \cite[Theorems 3.1 and 5.2]{MW}. We have that $\nu$ is injective for deformations of punctual Hilbert  schemes of $K3$ surfaces and $\OG 10$ manifolds, while the kernel is generated by the group of translations by points of order $n$ on an abelian surface and by $-\id$  (respectively, equal to $(\IZ/2\IZ)^{\times 8}$) if $\X$ is a deformation of a generalized Kummer of dimension $2n$ (respectively, of an $\OG 6$ manifold).  
\end{rmk}

\subsection{Basic facts  on Enriques manifolds}\label{sec:2.3}

Throughout this section, $X$
denotes an arbitrary Enriques manifold, $\X$  its IHS universal cover and $G$
its cyclic and finite fundamental group.
Now consider the quotient $\pi\colon \X\to X=\X/G$, where $|G|=d$. Recall that in general the  action of an automorphism $\phi\in \Aut(\X)$ of finite order is, respectively, symplectic, non-symplectic, or purely non-symplectic  if $\phi$ acts on the symplectic holomorphic 2-form by multiplication by, respectively, one, a  root of unity, or a root of unity of the same order as $\phi$.
In our case, observe that the action of any $g\in G$ must be  purely non-symplectic, 
as, if $G$  contained symplectic automorphisms, there would be  points with a non-trivial stabilizer and the covering would not be \'etale (see \cite[Section 2.2]{BNWS}).  Observe that if $G$ is of prime order, then a non-symplectic action is the same as a purely non-symplectic one. 
If we denote the generator of $G$ by $g$,  we have that $g^*\sigma=\lambda \sigma$ with $\lambda$ a primitive $\supth{d}$ root of unity.
Recall that the group of automorphisms (resp.\ of the birational transformations) of $X$ may be identified with the quotient by $G$ of the normalizer group of $G$ in $\Aut(\X)$ (resp.\ in $\Bir(\X)$); \textit{i.e.} 
\begin{equation}\label{eq:aut}
\Aut(X) =\left\{\t\in \Aut\left(\X\right): \t\circ G\circ \t^{-1}=G \right\}/G
\end{equation}
and 
\begin{equation}\label{eq:bir}
\Bir(X) =\left\{\t\in \Bir\left(\X\right): \t\circ G\circ \t^{-1}=G \right\}/G.
\end{equation} 
Indeed, if $\t\in \{\t\in \Aut(\X): \t\circ G\circ \t^{-1}=G \}/G$, then one obviously recovers an automorphism of $X$. For the other inclusion, if $\tau \in \Aut(X)$, then, by the universal property of the universal cover, the morphism $\tau\circ \pi$ factorizes through $\X$, namely there exists a morphism $\t:\X\to \X$ sitting in the following commutative diagram: 
$$
\xymatrix{
\X \ar[d]^\pi \ar[r]^\t &\X\ar[d]^\pi\\
X \ar[r]^\tau &X\rlap{.}}
$$
By construction $\t$ is bijective and yields the lifting of the automorphism $\tau$.
For birational transformations again one inclusion is obvious. For the other inclusion, if $\tau \in \Bir(X)$, then the largest open subsets $U_1, U_2 \subset X$ over which $\tau\colon U_1\rightarrow U_2$ is an isomorphism have complement of codimension at least $2$. The same is then true for $\pi^{-1}(U_1)=:\tilde U_1$ and $\pi^{-1}(U_2)=:\tilde U_2$. Since $\pi_1(\tilde U_j)=\pi_1(\tilde X)$ for $j=1,2$,   we have that $\tilde U_j$ for $j=1,2$ is simply connected, and the argument given before works the same and yields an isomorphism $\t\colon \tilde U_1\rightarrow \tilde U_2$ lifting $\tau$, which we see as $\t\in \Bir(\X)$. 

\begin{lem}\label{lem:inc}
For any Enriques manifold $X$ we have $H^2(X,\Z)= \Pic(X)$. Moreover, we have $\pi^*H^2(X,\Z)=H^2(\X,\Z)^G\subset \Pic(\X)$.  In particular, 
$$\dim_\R \pi^* N^1(X)_\R=\dim_\R \pi^* \Pic(X)_\R= \rk H^2\left(\X,\Z\right)^G.$$
\end{lem}

\begin{proof}
Consider $\xi\in H^2(X,\Z)$. We have 
$$
q(\pi^*\xi,\sigma)=q(g^*\pi^*\xi,g^*\sigma)=q(\pi^*\xi,\lambda \sigma)=\lambda q(\pi^*\xi, \sigma)
$$
with $\lambda\not=1$ as $g$ is not a symplectic automorphism. 
Then $q(\pi^*\xi,\sigma)=0$, which implies that $\pi^*\xi\in \Pic(\X)$, as $\Pic(\X)=\sigma^{\perp}\cap H^2(\X,\Z)$. In particular, we have that $\xi \in \Pic(X)$. We have shown that $H^2(X,\IZ)\subset \Pic(X)\subset H^2(X,\IZ)$, so we get the first equality (notice that this equality also follows immediately from the fact that $h^{2,0}(X)=0$). 
By construction $\pi^*H^2(X,\Z)\subset H^2(\X,\Z)^G$ with finite index,
and moreover by \cite[proofs of Propositions 2.8 and 5.1]{OS1} we in fact get equality. Finally, since $\pi^*\Pic(X)\subset \Pic(\X)$, the previous equality implies that $H^2(\X,\IZ)^G\subset \Pic(\X)$ (which is in fact a more general property of non-symplectic automorphisms acting on IHS manifolds, but this gives an easy proof).
\end{proof}

\subsection{Moduli spaces of marked Enriques manifolds}\label{S:moduli}
We follow the presentation in \cite{BCS}. See also \cite{OS2} for another equivalent approach and for the study of the period map.  Let $L$ be a lattice. Let $\X$ be an IHS manifold. Recall that a marking for $\X$ is an isometry $L\to H^2(\X,\Z)$.  Let $M$ be an even non-degenerate lattice of rank $\rho\geq 1$ and signature $(1,\rho-1)$.  An \emph{$M$-polarized} IHS manifold is a pair $(\X,j)$, where $\X$ is a projective IHS manifold and $j$ is a primitive embedding of lattices $j\colon M\hookrightarrow \NS(\X)$.  Two $M$-polarized IHS manifolds $(\X_1, j_1)$ and $(\X_2,j_2)$ are \emph{equivalent} if there exists an isomorphism $f\colon \X_1\to \X_2$ such that $j_1=f^*\circ j_2$. As in \cite[Section 10]{DK} and \cite{Dol96}, one can construct a moduli space of \emph{marked} $M$-polarized IHS manifolds as follows. We fix a primitive embedding of $M$ in $L$, and we identify $M$ with its image in $L$. A \emph{marking} of $(\X,j)$ is an isomorphism of lattices $\eta\colon L \to H^2(\X,\IZ)$
such that $\eta_{|M}=j$. As observed in \cite[Section~5, p.~2609]{Dol96}, 
if the embedding of $M$ in $L$ is unique up to an isometry of $L$, then every $M$-polarization admits a compatible marking.  Two $M$-polarized marked IHS manifolds $(\X_1,j_1, \eta_1)$ and $(\X_2, j_2, \eta_2)$ are \emph{equivalent} if there exists an isomorphism $f\colon \X_1\to \X_2$ such that $\eta_1=f^*\circ \eta_2$ (this clearly implies that $j_1=f^*\circ j_2$). Let $T:=M^{\perp}\cap L$ be the orthogonal complement of $M$ in $L$, and set
$$
\Omega_{M}:=\{x  \in \IP\left(T_\IC\right)\mid q(x)=0,\, q(x+\bar{x})>0\}.
$$   
Since $T$ has signature $(2,\rk(T)-2)$,  the period domain $\Omega_{M}$ is a disjoint union of two connected components of  dimension $\rk(T)-2$. For each $M$-polarized marked IHS manifolds $(X,j,\eta)$, 
since $\eta(M)\subset \NS(\X)$, we have $\eta^{-1}(H^{2,0}(\X))\in\Omega_M$. On the other hand, by the surjectivity of the period map, see \cite[Theorem~8.1]{Huy-basic}, 
restricted to any connected component $\mathfrak{M}^0_L$ of the moduli space $\mathfrak{M}_L$  of IHS manifolds with second integral cohomology isometric to a given lattice $L$, we can associate to each point $\omega\in\Omega_M$ an $M$-polarized IHS manifold as follows:
there exists a marked pair $(\X,\eta)\in\frM_L^0$ such that $\eta^{-1}(H^{2,0}(\X))=\omega \in \IP(T_\IC)$, so $M=T^{\perp}\cap L\subset \omega^{\perp}\cap L$; hence $\eta(M)\subset H^{2,0}(\X)^{\perp}\cap H^2(\X,\Z)=\NS(\X)$, and we take $(\X,\eta_{|M},\eta)$. 
By the local Torelli theorem for IHS manifolds, an $M$-polarized IHS manifold  $(\X,j)$ has a local deformation space $\Def_M(\X)$ that is contractible, smooth of dimension $\rk(T)-2$ and such that the (local) period map ${\mathfrak p }\colon\Def_M(\X)\to\Omega_M$ is a local isomorphism (see \cite{Dol96}). By gluing all these local deformation spaces, one obtains a moduli space $K_M$ of marked $M$-polarized IHS manifolds that is a non-separated 
analytic space, with a (global) period map $\mathfrak p \colon K_M\to\Omega_M$. 
To construct a period domain for Enriques manifolds, we   take the non-symplectic action into account. 
Let $(\X,j)$ be an $M$-polarized IHS manifold and $G=\langle g \rangle$ be a cyclic group of  
order $d\geq 2$ acting purely non-symplectically on $\X$. 
Assume that the action of $G$ on $j(M)$ is the identity
and that there exists a group homomorphism $\rho\colon G\to O(L)$
such that
$$
M=L^{\rho}:=\{x\in L\mid \rho(g)(x)=x,\forall g\in G\}.
$$
We define a {\it $(\rho, M)$-polarization} of $(\X,j)$ as the data of a marking $\eta\colon L\to H^2(\X,\IZ)$ such that $\eta_{|M}=j$ and $g^\ast=\eta\circ\rho(g)\circ\eta^{-1}$.  Two $(\rho, M)$-polarized IHS manifolds $(\X_1, j_1)$ and $(\X_2,j_2)$ are {\it isomorphic} if there are markings $\eta_1\colon L\to H^2(\X_1, \IZ)$ and $\eta_2\colon L\to H^2(\X_2,\IZ)$
such that $\eta_i{_{|M}}=j_i$, and an isomorphism $f\colon \X_1\to \X_2$ such that $\eta_1=f^*\circ \eta_2$.  Recall that by construction $\IC\sigma$ is identified with the line in $L_\mathbb{C}$ defined by $\eta^{-1}(H^{2,0}(\X))$.  Let $\lambda\in\IC^*$ be such that $\rho(g)(\sigma)=\lambda \sigma$. Observe that $\lambda\not= 1$ since the action is purely non-symplectic and $\lambda$ is a primitive $\supth{d}$ root of unity.  By construction $\sigma$ belongs to the eigenspace of $T_\IC$ relative to the eigenvalue $\lambda$, where $T=M^{\perp}\cap L$. We denote it by $T(\lambda)$ (if $d=2$, we have $\lambda=-1$ and we write $T(\lambda)=T_{\IR}(\lambda)\otimes \IC$, where $T_{\IR}(\lambda)$ is the real eigenspace relative to $\lambda=-1$ and remark that in this case $T(\lambda)=T_\IC$).  Assume that $\lambda\not= -1$; then the period belongs to the space
$$
\Omega_M^{\rho,\lambda}:=\{x\in \IP(T(\lambda))\mid q(x+\bar{x})>0\}
$$ 
of dimension $\dim T(\lambda)-1$, which is a complex ball if $\dim T(\lambda)\geq 2$. By using the fact that $\lambda\not=-1$, it is easy to check that every point $x\in \Omega_M^{\rho,\lambda}$  automatically satisfies the condition $q(x)=0$.  If $\lambda=-1$, we set $\Omega_M^{\rho,\lambda}:=\{x\in \IP(T(\lambda))\mid q(x)=0, \, q(x+\bar{x})>0\}$.  It has dimension $\dim T(\lambda)-2$. Clearly, in both cases $\Omega_M^{\rho,\lambda}\subset \Omega_M$.

Let $X$ be a {\it marked} Enriques manifold of index $d$, \textit{i.e.}~a {\it $(\rho, M)$-polarization} of $(\X,j)$ with a marking $\eta\colon L\to H^2(\X,\IZ)$ such that $\eta_{|M}=j$ and $g^\ast=\eta\circ\rho(g)\circ\eta^{-1}$, so that $X=\X/G$.  Recall that by the Bogomolov--Tian--Todorov theorem, the Kuranishi space $\Def(X)$ of an Enriques manifold $X$ is smooth. Moreover, after possibly shrinking $\Def(X)$, we have that every point in it parametrizes an Enriques manifold of the same index; see \cite[Proposition 1.2]{OS2}.  If $\mathcal X\to B$ is a flat family of marked Enriques manifolds over some simply connected complex space $B$, then, as remarked in \cite[Section 2]{OS2}, the universal covering $\mathcal \X\to \mathcal X$ of the family is also the fiberwise universal covering. The period map of the family $\mathcal X\to B$ is then defined as
\begin{equation}
{\mathfrak p}_B \colon B\lra \Omega_M^{\rho,\lambda},\quad b\longmapsto \eta^{-1}\left(H^{2,0}\left(\X_b\right)\right).
\end{equation}
By \cite[Theorem 2.4]{OS2} the local Torelli theorem holds; namely, the period map ${\mathfrak p}_B$ is a local isomorphism.  Hence, as for IHS manifolds, we can patch together the Kuranishi spaces via the Oguiso--Schr\"oer local Torelli theorem, to construct a (non-separated) moduli space of (marked) Enriques manifolds in terms of $(\rho, M)$-polarized IHS manifolds.  First, we prove the following.

\begin{lem}\label{lem:perp}
  We have $H^2(\X,\Q)= \pi^* [H^2(X,\Q)] \oplus [\pi^*(H^2(X,\Q)) ]^{\perp}$. In particular, the lattice  $\pi^* [H^2(X,\Z)]=H^2(\tilde X,\Z)^G$ is non-degenerate.
\end{lem}

\begin{proof} Clearly, both spaces $\pi^* [H^2(X,\Q)]$ and  $[\pi^*(H^2(X,\Q)) ]^{\perp}$
are contained in $H^2(\X,\Q)$. We need to show that their sum is direct. Let us take a non-zero class $\beta \in [\pi^*(H^2(X,\Q))]^{\perp}$.  Suppose towards a contradiction that
\begin{equation}\label{eq:1H}
  \beta \in \pi^* H^2(X,\Q). 
\end{equation} 
Then $q(\beta, \beta)=0$.  Let $B\in \Pic(X)$ be an ample line bundle, and consider the ample line bundle $A:=\pi^*B$ on $\X$. By the choice of $\beta$ and  $A$, we have that $q(\beta,A)=0$, which, by Equation~(\ref{eq:1H}), contradicts the Hodge index theorem for $q$ (notice that since $h^{2,0}(X)=0$, we have $\beta\in H^{1,1}(\X)$).
\end{proof}

We recall that a {\it very general Enriques manifold} is a point in the corresponding parameter space lying  outside a countable union of proper closed subsets. When the index is prime, the following holds.

\begin{prop}\label{prop:gen}  
 For a very general Enriques manifold $X$  of prime index $|G|=p$ with universal cover $\X$, we have  that $\Pic(\X)= H^2(\X,\Z)^G$. 
\end{prop}
 
\begin{proof}
Our argument is similar to that in \cite[Proposition 5.3 and Theorem 5.5]{BCS}. As in Section~\ref{S:moduli} we set $ M:=\eta^{-1} (H^2(\tilde{X}, \IZ)^G)$ and $T:=M^{\perp}\cap L$.  Let $T^+ \subset T\backslash \{0\}$ be the set of $t\in T\backslash \{0\}$ such that $H_t=\{\omega\in \Omega^{\rho,\lambda}_M\mid q( \omega, t)=0\}$ is a hyperplane section; more precisely, this is a proper subset of $\Omega^{\rho,\lambda}_M$.  Consider $\mathcal{H}:=\bigcup_{t\in T^+} H_{t}$. Each subset $\Omega^{\rho,\lambda}_M\setminus{H}_{t}$ is open and dense in $\Omega^{\rho,\lambda}_M$; hence by Baire's theorem the subset $\Omega^0:=\Omega^{\rho,\lambda}_{M}\setminus\mathcal{H}$ is dense in $\Omega^{\rho,\lambda}_M$ since~$\mathcal{H}$ is a countable union of complex proper closed subspaces.  We now take  a period $\omega \in \Omega^0$ and a marked IHS manifold $(\X,\eta)$ such that $\mathfrak p (\X,\eta)=\omega$. To lighten the notation, from now on we will omit the marking $\eta$. Then $\Pic(\tilde
X)=\{l\in L:q(\omega, l)=0\}$.  Observe that since $\omega \in T(\lambda)$, we have that $H^2(\tilde X,\IZ)^G\subset \Pic(\tilde X)$, as by construction $ M:=H^2(\tilde{X}, \IZ)^G$ is orthogonal to $T$. Moreover, by Lemma~\ref{lem:inc} we have the equality $\Pic(\tilde X)^G=H^2(\tilde X,\IZ)^G$.  We now want to show that
\begin{equation}\label{pic_int}
\left(\Pic\left(\tilde X\right)^G\right)^\perp \cap \Pic\left(\tilde X\right)=\{0\}.
\end{equation}
The fact that (\ref{pic_int}) implies that
$\Pic(\tilde X)=H^2(\tilde X,\IZ)^G$ follows from the equality  $L_\Q=(\Pic(\tilde X)^G)_\Q\oplus(\Pic(\tilde X)^G)^\perp_\Q$, which we proved in Lemma~\ref{lem:perp}. 
We now show Equality (\ref{pic_int}). Let $0\neq t\in (\Pic(\tilde X)^G)^\perp \cap \Pic(\tilde X)$. 
Then $t^\perp\subset L_{\C}$ cannot determine a hyperplane section of $\Omega^{\rho,\lambda}_M$ as we have taken $\omega \in \Omega^0$. In particular,  this implies that, for all $t\in  (\Pic(\tilde X)^G)^\perp \cap \Pic(\tilde X))$, the subspace $t^\perp$ contains  the eigenspace $T(\lambda)$ with respect to $\lambda$ for the action of $G$ on $T_\IC$. Hence 
$$
T(\lambda) \subset \left[\left(\Pic\left(\tilde X\right)^G\right)^\perp \cap \Pic\left(\tilde X\right)\right]^\perp _\C,
$$
which implies that
$$
T(\lambda) \subset \left(\Pic\left(\X\right)^G + \Pic\left(\X\right)^\perp\right)_\C. 
$$
Now write any $v_\lambda\in T(\lambda)$ as $v_\lambda=v_1+v_2$ with $v_1\in \Pic(\X)^G_\C$ and $v_2\in \Pic(\X)^\perp_\C\subset (\Pic(\X)^G)^\perp_\C$. By construction we also have $T(\lambda)\subset (\Pic(\X)^G)^\perp_\C$. So $v_\lambda-v_2=v_1$ is an element in $((\Pic(\X)^G)^\perp\cap (\Pic(\X)^G))_\C=\{0\}$ by Lemma~\ref{lem:perp}. This shows that $ T(\lambda) \subset \Pic(\X)^\perp_\C=:T_{\tilde X,\C} $, where the orthogonal complement of $\Pic(\tilde X)$ is the transcendental lattice $T_{\tilde X}$ of $\tilde X$.  This in particular implies
\begin{equation}\label{eq:contr}
    \left(\left(\Pic\left(\tilde X\right)^G\right)^\perp \cap \Pic\left(\tilde X\right)\right)_\IC\cap T(\lambda)=\{0\}; 
\end{equation}
otherwise, we would find a non-zero element in $\Pic(\tilde X)\cap T_{\tilde X}$, which  is not possible. 
Now by definition $G$ acts as the identity on $\Pic(\tilde X)^G=H^2(\tilde X,\IZ)^G$, while, by the hypothesis $|G|=p$, the eigenvalues of its  action on $ \left((\Pic(\tilde X)^G)^\perp \cap \Pic(\tilde X)\right)_\IC$ are primitive roots of unity. Since this intersection is a lattice, the characteristic polynomial is a power of the $\supth{p}$ cyclotomic polynomial (as the characteristic polynomial has integral coefficients).  This means that all the primitive roots of unity appear with the same multiplicity, which is different from zero since we are assuming that $t\neq 0$.  In particular, $((\Pic(\tilde X)^G)^\perp \cap \Pic(\tilde X))_\IC\cap T(\lambda)\neq \{0\}$, which contradicts Equation~(\ref{eq:contr}).  \color{black} So we must have that $(\Pic(\tilde X)^G)^\perp \cap \Pic(\tilde X)=\{0\}$, and this concludes the proof.
\end{proof}

\begin{rmk}
Observe that, if the integer $|G|=d$ is not a prime number, then  at all points $[\X]$ in the moduli space we have  $H^2(\X,\IZ)^G\subset \Pic(\X)$ by Lemma~\ref{lem:inc}, and we may or may not have  equality at the very general point;  
see the  discussion of the index 4 examples in Section~\ref{ss:notprime}. 
\end{rmk}

\section{Proof of Theorem~\ref{identity}}\label{sec:proof}
Throughout this section, $X$ denotes an arbitrary Enriques manifold, $\X$ its IHS universal cover, $G$ its fundamental group and $\pi\colon\X\to X=\X/G$ the quotient morphism.  We assume the hypothesis of Theorem~\ref{identity} and show that $\tilde\tau^*$ commutes with $g^*$, for any $\tilde\tau\in \Bir(\X)$.

\begin{lem}\label{lem:on-G-inv}
Each  $\tilde\tau\in \Bir(\X)$  
acts on  $H^2(\X,\Z)^G$ 
and commutes with the action of $g$ on it.
\end{lem}

\begin{proof} 
By assumption we know that $\Pic(\X)= H^2(\X,\Z)^G$.  Consider $\eta\in H^2(\X,\Z)^G$ and set $e:=(\tilde\tau)^*(\eta)$.  First notice that $e\in H^2(\X,\Z)^G$. Indeed, birational automorphisms preserve the Picard group; hence $e=(\tilde\tau)^*(\eta)\in \Pic(\X)=H^2(\X,\Z)^G$. Therefore, $g^*(e)=e$ and
$$
g^*((\tilde \tau)^*(\eta))=g^*(e)=e= (\tilde \tau)^*(\eta)= (\tilde\tau)^* (g^*(\eta)), 
$$
where the last equality holds since $\eta\in H^2(\X,\Z)^G$. 
\end{proof}

We then also check the commutativity on $T_\X:=(\pi^*H^2(X,\Z))^{\perp}\cap H^2(\X,\Z)$.

\begin{lem}\label{lem:on-transc} 
  Each  $\tilde\tau\in \Bir(\X)$
acts on  $T_\X$ and commutes with the action of $g$ on it. 
\end{lem}

\begin{proof}
Notice that $T_\X$ is the transcendental part of a weight 2 Hodge structure of $K3$ type. Hence by \cite[Chapter 3, Corollary 3.4]{Huy-K3} the group of all Hodge isometries of $T_\X$ is a finite cyclic group.  Birational automorphisms preserve $\Pic(\X)$, which, by hypothesis, equals $H^2(\X,\Z)^G$. Therefore, any element of $\Bir(\X)$ preserves $T_\X$ and induces a Hodge isometry of $T_\X$. Then the restrictions of $(\tilde\tau)^*$ and $g^*$ to $T_\X$ are elements of a finite cyclic group, and as such they commute.
\end{proof}

\begin{proof}[Proof of Theorem~\ref{identity}]
Consider $\xi\in \Nef^+(X)$. Then there exists an $\eta \in \Nef^+(\tilde X)$ such that $\xi=\pi_*(\eta)$ (take \textit{e.g.}~$\eta=(\pi^*\xi)/|G|$).  By \cite{AV1,AV2} there exists a fundamental domain $\D$ for the $\Aut(\X)$-action on $\Nef^+(\X)$, which is a rational polyhedral convex cone.  In particular, we have the existence of $\tilde\tau\in \Aut(\X)$ and $\tilde\delta\in \tilde D\cap \pi^*N^1(X)_\R=:D$ such that $(\tilde\tau)^*(\eta)= \tilde\delta$ (observe that $\tilde\delta\in \pi^*N^1(X)_\R$ as by hypothesis and by Lemma~\ref{lem:inc} we have $\pi^* N^1(X)_\mathbb{R}=H^2(\tilde X,\mathbb{Z})_\mathbb{R}^G=\Pic(\tilde X)_\mathbb{R}$, so $\tilde\tau $ sends $\eta$ to $\pi^* N^1(X)_\mathbb{R}$).  Notice that $D$ is a rational polyhedral convex cone, since $\tilde D$ is and $\pi^*N^1(X)_\R$ is a rational subspace.  Since $\Pic(\tilde X)=H^2(\tilde X,\IZ)^G$, we can then apply Lemmas~\ref{lem:on-G-inv},~\ref{lem:on-transc} and~\ref{lem:perp} to show that the isometry $\tilde\tau^*$ commutes with $g^*$ on $H^2(\X,\Z)$.  If the map $\Aut(\X)\rightarrow O(H^2(\X,\IZ))$ is injective, the proof is easier; we give it explicitly for the reader's convenience. Indeed, in this case $\tilde \tau$ and $g$  commute  on $\X$ as well, which by Equation~(\ref{eq:aut}) implies that $\tilde\tau$ descends to an automorphism $\tau$ on $X$ such that $ \tau\circ \pi=\pi \circ\tilde \tau.  $ So we get
$$
\tau^*(\xi)=\tau^* (\pi_*(\eta))=\pi_*\left(\left(\tilde\tau\right)^* (\eta)\right)=\pi_*\left(\tilde \delta\right)\in \pi_*(D).
$$
This means that $ \pi_*(D)$ is a  fundamental domain for the $\Aut (X)$-action on $\Nef^+(X)$, which moreover is a rational and polyhedral convex cone as $D$ is.
In general, if the map $\Aut(\X)\rightarrow O(H^2(\X,\IZ))$ is not injective, then by \cite[Proposition 9.1(v)]{Huy-basic} its kernel $K$ is finite. Consider 
$$
\Gamma:=\left\{\varphi \in \Aut\left(\X\right)\mid \varphi g^{-1}\varphi^{-1} g\in K\right\}, 
$$
and observe that $\Gamma$ is in fact equal to $\Aut(\X)$ as we have shown that
 each automorphism induces an isometry on $H^2(\X,\IZ)$ that commutes  with $g$. Consider 
 $$
 \Gamma_0:=\left\{\varphi \in \Aut\left(\X\right)\mid \varphi g^{-1}\varphi^{-1} g=\id \right\}
 $$ 
 and notice that $\Gamma_0$ is a subgroup of  $\Aut(\X)$ which is not necessarily normal. Let us show that it has finite index in $\Aut(\X)$. Consider the map 
 $$
 f \colon \Aut\left(\X\right)\longrightarrow K, \quad \varphi\longmapsto \varphi g^{-1}\varphi^{-1} g.
 $$
Then one easily checks that $f(\varphi_1)=f(\varphi_2)$ if and only if $\varphi_2^{-1}\varphi_1\in \Gamma_0$. Therefore, if we take the set $\Aut(\X)/\Gamma_0$ of left cosets, we have an injective map $\Aut(\X)/\Gamma_0\rightarrow K$ showing that $\Gamma_0$ has finite index in $\Aut(\X)$.  Now let  $\bar \gamma_1,\ldots,\bar \gamma_r$ be all the classes in $\Aut(\X)/\Gamma_0$ and $\varphi$ be an automorphism such that $\varphi^*(\pi^*\xi)\in \tilde D$. By considering the class $\bar\varphi\in\Aut(\X) / \Gamma_0$, we then have that $\bar \varphi=\bar \gamma_j$ for a certain $j=1,\ldots,r$ such that $\gamma_j^{-1}\varphi\in\Gamma_0$. We then modify the fundamental domain $\tilde D$ by taking
$$
\tilde D ':=\tilde D \cup \bigcup_{i=1}^r \left(\gamma_i^{-1}\right)^*\left(\tilde D\right).
$$
Observe that $\tilde D '$  is obviously still rational and polyhedral and $(\gamma_j^{-1})^*\varphi^*(\pi^*\xi)\in \tilde D'$. Moreover, $\gamma_j^{-1}\varphi$ descends to an automorphism on $X$, and we set $\tilde D'\cap \pi^*N^1(X)_{\mathbb R}=:D'$. 
If $\Aut^*(X)$ denotes the image of $\Aut(X)$ inside $\GL(N^1(X)_ {\mathbb R})$, then $(\Nef^+(X), \Aut^*(X))$ is of polyhedral type (see Section~\ref{subsec:cones}): it suffices to take $\Pi$ to be  the convex hull of $ \pi_*(D')$ inside $N^1(X)_{\mathbb R}$. Then by Looijenga's result (Lemma~\ref{lem:loo}) applied to $(\Nef^+(X), \Aut^*(X))$, we deduce the existence of a rational polyhedral cone which is a  fundamental domain for the $\Aut(X)$-action on $\Nef^+(X)$.

For the existence of a  fundamental domain for the $\Mov^+(X)$-action on $\Bir(X)$, the same arguments apply, using Equation~(\ref{eq:bir}) instead of Equation~(\ref{eq:aut}), the birational cone conjecture for IHS manifolds proved in \cite{Mark-survey} and noticing that the kernel of $\Bir(\X)\to O(H^2(\X, \mathbb Z))$ is again finite, by \cite[Proposition 9.1(iv) and (v)]{Huy-basic}. 
\end{proof}

\begin{rmk}
Observe that when the map $\Aut(\X)\rightarrow O(H^2(\X,\IZ))$ is injective, we have shown that the centralizer of $G$ in $\Aut(\X)$ (resp.\ $\Bir(\X)$) is exactly equal to $\Aut(\X)$ (resp.\ to $\Bir(\X)$), so   $\Aut(X)$ (resp.\ $\Bir(X)$) is the quotient of $\Aut(\X)$ (resp.\ of $\Bir(\X)$) by $G$.
\end{rmk}

\section{Further remarks and results}

In Section~\ref{ss:notprime} we recall from \cite{BNWS,OS1} the constructions of all the currently known examples of Enriques manifolds 
and prove that for these examples the fundamental group acts trivially on the Picard group of the  universal cover. In Section~\ref{sec:4.2} we check this condition for certain indices~$d$ (some of which even non-prime) at all points of the moduli space. 

\subsection{The cone conjecture for the known examples}\label{ss:notprime}
We start with the index 2 examples. 

\subsubsection*{Index 2: A quotient of the Hilbert scheme}
Let $S$ be a K3 surface with an Enriques involution $\iota$. For a generic such K3 surface it is well known that $\Pic(S)=H^2(S,\IZ)^\iota=U(2)\oplus E_8(2)$. The quotient $S^{[n]}/\langle \iota^{[n]}\rangle$, for $n$ odd, is an Enriques manifold of index $2$. Recall that $\Pic(S^{[n]})=\Pic(S)\oplus \Z \delta$, where $\delta$ is half of the class of the exceptional divisor on $S^{[n]}$. Since $\iota^{[n]}$ is a natural automorphism, its action on $\delta$ is the identity, and the action on $\Pic(S)$ is the same as the action of $\iota$, so  $\Pic(S^{[n]})=H^2(S^{[n]},\IZ)^{\iota^{[n]}}$. In other words, we have the following. 

\begin{prop}\label{prop:ind2_hilb}
Let $S^{[n]}/\langle \iota^{[n]} \rangle$ be the index 2 Enriques manifold recalled above.  Then $\Pic(S^{[n]})=H^2(S^{[n]},\IZ)^{\iota^{[n]}}$.
\end{prop}

\subsubsection*{A quotient of a generalized Kummer}
Let $A=E\times F$ be the product of two elliptic curves, and assume that $n$ is odd so that we can write $2m=n+1$ for an integer $m$. For $E,F$  very general elliptic curves, we have $\rk\Pic(E\times F)=2$; indeed, 
\begin{equation}\label{eq:Pic}
\Pic(E\times F )=\Pic(E)\times \Pic(F)\times \Hom(\Jac(E),\Jac(F)),
\end{equation} 
and for the very general choice of $E$ and $F$,  we have 
\begin{equation}\label{eq:Fgen}
\Hom(\Jac(E),\Jac(F))=\{0\}.
\end{equation}
Now consider  $a:=(a_1,a_2)$, where $a_2\in F$ is a $2$-torsion point and $a_1\in E[n+1]$ is such that $m a_1\not\in \IZ\oplus t\IZ$ (where $E=\IC/\IZ\oplus t\IZ$). Let $t_a$ be the translation by the point $a$ on $E\times F$ and $h_2=\diag(-1, 1)$ the morphism given by the multiplication by $-1$ on the first component and the identity on the second. Set $\psi_2:=t_a\circ h_2$. Then $\psi_2^{[n]}$ does not have fixed points on $\Km_n(E\times F)$, as shown in \cite[Section 6]{OS1}. 

\begin{prop}\label{prop:ind2} Let 
$\Km_n(E\times F)/\langle \psi_2^{[n]} \rangle$ be the index~$2$ Enriques manifold recalled above.
    Then we have $\Pic(\Km_n(E\times F))= H^2(\Km_n(E\times F), \IZ)^{\psi_2^{[n]}}$.
\end{prop}

\begin{proof}
    For $(z,w)\in E\times F$ we have $t_a(h(z,w))=(-z+a_1,w+a_2)$. Let us now compute the action in cohomology. Recall that $H^2(A,\IZ)=U\oplus U\oplus U$, and consider $H^2(A,\IR)$. If we write $z=z_1+iz_2$ for the coordinate on $E$ and $w=w_1+iw_2$ for the coordinate on $F$, then $H^2(A,\IR)$ is generated by the 2-forms
    $$
    dz_1\wedge dw_1,\; dz_1\wedge dw_2,\; dz_2\wedge dw_1,\; dz_2\wedge dw_2,\; dz_1\wedge dz_2,\; dw_1\wedge dw_2.
    $$
    A translation acts trivially in cohomology, and the multiplication by $-1$ acts by sending $z_1+iz_2$ to $-z_2-iz_1$, so the image under $\psi_2^*$ of the previous basis is
$$\psi_2^*(dz_1\wedge dw_1)=-dz_1\wedge dw_1,\; \psi_2^*(dz_1\wedge dw_2)=-dz_1\wedge dw_2,\; \psi_2^*(dz_2\wedge dw_1)=-dz_2\wedge dw_1, $$ 
$$\psi_2^*(dz_2\wedge dw_2)=-dz_2\wedge dw_2,\; \psi_2^*(dz_1\wedge dz_2)=dz_1\wedge dz_2,\; \psi_2^*(dw_1\wedge dw_2)=dw_1\wedge dw_2.$$ 
Hence the matrix of $\psi_2^*$ in this basis is
$$
\left(
\begin{array}{rrrrrr}
-1&0&0&0&0&0\\
0&-1&0&0&0&0\\
0&0&-1&0&0&0\\
0&0&0&-1&0&0\\
0&0&0&0&1&0\\
0&0&0&0&0&1\\
\end{array}
\right).
$$
The eigenvalues of this matrix are $-1$ with multiplicity $4$ and $1$ with multiplicity $2$. This means that the invariant part for the action of $\nu$ on the cohomology of $A$ is precisely $\Pic(E\times F)$.  Consider the induced natural automorphism $\psi_2^{[n]}$ on $\Km_n(E\times F)$. By the choice of $a_1$ and  $a_2$, the  induced automorphism $\psi_2^{[n]}$ acts freely on $\Km_n(E\times F)$ (see \cite[Theorem 6.4]{OS1}).
From the above description, we have that the invariant part for the action on the second cohomology of $\Km_n(E\times F)$ with integral coefficients of the induced automorphism $\psi_2^{[n]}$  is  exactly the Picard group (since the automorphism $\psi_2^{[n]}$ is natural, it acts as the identity on the exceptional divisor). 
\end{proof}

\subsubsection*{Index 3: A quotient of a generalized Kummer} Take  integers $m,n$ such that $3m=(n+1)$. Let $E_\omega$ be an elliptic curve with complex multiplication by $\omega$, which is a primitive $\suprd{3}$ root of unity, and let $F$ be another elliptic curve.  Set $A:=E_\omega\times F$.  For $F$ a very general elliptic curve we have $\rk\Pic(E_\omega\times F)=2$, by (\ref{eq:Pic}) and (\ref{eq:Fgen}). Now consider $a:=(a_1,a_2)$, where $a_2\in F$ is a $3$-torsion point and $a_1\in E_\omega[n+1]$ is such that $m a_1(2+\omega)\not\in \IZ+\omega \IZ$.  Let $t_a$ be the translation by the point $a$ on $E_\omega\times F$ and $h_3=\diag(\omega, 1)$ the morphism given by the multiplication by $\omega$ on the first component and the identity on the second. Set $\psi_3:=t_a\circ h_3$. Then $\psi_3^{[n]}$ does not have fixed points on $\Km_n(E_\omega\times F)$, as shown in \cite[Section 4.2]{BNWS} and \cite[Section 6]{OS1}.

\begin{prop}\label{prop:ind3} Let 
$\Km_n(E_\omega\times F)/\langle \psi_3^{[n]} \rangle$ be the index 3 Enriques manifold recalled above.
    Then
        $\Pic(\Km_n(E_\omega\times F))= H^2(\Km_n(E_{\omega}\times F),\IZ)^{\psi_3^{[n]}}$.
\end{prop}

\begin{proof}
To show that $\psi_3^{[n]}$ acts as the identity on $\Km_n(E_\omega\times F)$, we argue here in a different way than in the index 2 example.  Observe that $E_\omega\times F$ is projective, so  it contains an invariant ample class. If the eigenspace relative to $\omega$ were in $\Pic(E_\omega\times F)_\mathbb{C}$, then the same would be true for $\bar \omega$.  Since $\rk\Pic(E_\omega\times F)=2$, this would lead to a contradiction,  so $\psi_3$ acts as the identity on the Picard group. Now $\psi_3^{[n]}$ acts as the identity on the exceptional divisor on $\Km_n(E_\omega\times F)$. This means that the action of $\psi_3^{[n]}$ on $\Pic(\Km_n(E_\omega\times F))$ is the identity.
\end{proof}

\subsubsection*{Index 4: A quotient of a generalized Kummer} 
Take integers $m,n$ such that $4m=(n+1)$. Let $E_i$ be an elliptic curve with complex multiplication by $i$, and let $F$ be another elliptic curve. Set $A:=E_i\times F$.  For $F$ a very general elliptic curve, we have $\rk\Pic(E_i\times F)=2$, by Equations~(\ref{eq:Pic}) and (\ref{eq:Fgen}). Now consider $a:=(a_1,a_2)$, where $a_2\in F$ is a $4$-torsion point and $a_1\in E_i[n+1]$ is such that $2ma_1(1+i)\not\in \mathbb \Z+i\Z$.  Let $t_a$ be the translation by the point $a$ on $E_i\times F$ and $h_4=\diag(i , 1)$ the morphism given by the multiplication by~$i$ on the first component and the identity on the second. Set $\psi_4:=t_a\circ h_4$. Then $\psi_4^{[n]}$ does not have fixed points on Kum$_n(E_i\times F)$, as shown in \cite[Section 4.2]{BNWS} and \cite[Section 6]{OS1}.

\begin{prop}\label{prop:ind4}
Let $\Km_n(E_i\times F)/\langle \psi_4^{[n]} \rangle$ be the index 4 Enriques manifold recalled above.  Then we have $\Pic(\Km_n(E_i\times F))= H^2(\Km_n(E_i\times F),\IZ)^{\psi_4^{[n]}}$.
\end{prop}

\begin{proof}
    For $(z,w)\in E_i\times F$, we have $t_a(h(z,w))=(iz+a_1,w+a_2)$. Let us now compute the action in cohomology. We use the same set of generators of $H^2(A,\IR)$ as in the index 2 case.  A translation acts trivially in cohomology, and the multiplication by $i$ acts by sending $z_1+iz_2$ to $-z_2+iz_1$, so  the image under $\psi_4^*$ of the previous basis is
$$\psi_4^*(dz_1\wedge dw_1)=-dz_2\wedge dw_1,\; \psi_4^*(dz_1\wedge dw_2)=-dz_2\wedge dw_2,\; \psi_4^*(dz_2\wedge dw_1)=dz_1\wedge dw_1, $$ 
$$\psi_4^*(dz_2\wedge dw_2)=dz_1\wedge dw_2,\; \psi_4^*(dz_1\wedge dz_2)=dz_1\wedge dz_2,\; \psi_4^*(dw_1\wedge dw_2)=dw_1\wedge dw_2.$$ 
Hence the matrix of $\psi_4^*$ in this basis is
$$
\left(
\begin{array}{rrrrrr}
0&0&1&0&0&0\\
0&0&0&1&0&0\\
-1&0&0&0&0&0\\
0&-1&0&0&0&0\\
0&0&0&0&1&0\\
0&0&0&0&0&1\\
\end{array}
\right).
$$
The eigenvalues of this matrix are $i$ and $-i$ with multiplicity $2$ and $1$ with multiplicity 2. This means that the invariant part for the action of $\psi_4$ on the cohomology of $A$ is precisely $\Pic(E_i\times F)$.  Consider the induced natural automorphism $\psi_4^{[n]}$ on $\Km_n(E_i\times F)$. By the choice of $a_1$ and $a_2$, the  induced automorphism $\psi_4^{[n]}$ acts freely on $\Km_n(E_i\times F)$ (see \cite[Theorem 6.4]{OS1}).
From the above description, we have that the invariant part for the action of the induced automorphism $\psi_4^{[n]}$ on the second cohomology of $\Km_n(E_i\times F)$ with integral coefficients  is  exactly the Picard group (since the automorphism $\psi_4^{[n]}$ is natural, it acts as the identity on the exceptional divisor).
\end{proof}

\begin{rmk}
Observe that the argument that we used for the index 3 case does not apply for the index 2 and the index 4 cases, because $-1$ could be an eigenvalue of $\psi_2$, respectively $\psi_4$, and so we cannot deduce immediately that the action of the automorphism is the identity on $\Pic(A)$.
\end{rmk}

\subsection{The cone conjecture for other possible cases}\label{sec:4.2}
We start by recalling a general result stated in \cite[Proposition 2.4]{OS1}. 

\begin{prop}\label{prop:d}
Let $X$ be an Enriques manifold of dimension $\dim(X)=2n$. Then the index $d$ of\, $X$ divides $n+1$.
\end{prop}

We now  discuss Enriques manifolds that may arise from all the known examples of IHS manifolds.

\subsection*{$\boldsymbol{K3^{[n]}}$ and $\boldsymbol{\Km_n}$}
Consider an Enriques manifold $\X\to X=\X/G$ such that $\X$ is a deformation of $K3^{[n]}$ or of $\Km_n$. By Proposition~\ref{prop:d} the order $d$ of the group $G=\langle g \rangle$ divides $n+1$. Recall that the second Betti numbers of $K3^{[n]}$ and of $\Km_n$ manifolds are, respectively, $23$ and $7$. Now since some  primitive $\supth{d}$ roots of unity are eigenvalues of the matrix of the action of $g^*$ on $H^2(\X,\IC)$ (see Section~\ref{sec:2.3}),  the Euler totient function must satisfy $\varphi(d)\leq 22$,  respectively $\varphi(d)\leq 6$ (in fact,  in general for an IHS manifold $\X$ we have $\varphi(d)\leq b_2(\X)-1$). A list of possible $d$ is given in \cite[Proposition 2.9]{OS1}; however, notice that the authors missed the cases $d=48,\,60$ for $K3^{[n]}$, and they erroneously included $d=24$ for $\Km_n$. Indeed, this is not possible since $\varphi(24)=8$, which is bigger than $b_2(\X)=7$. For convenience we recall in the next proposition all the possible values of $d$ and explain when the cone conjecture holds.

\begin{prop}\label{prop:other} Let $\X\to X=\X/G$ be an Enriques manifold quotient of an IHS manifold $\X$. 
\begin{enumerate}[label={\rm(\alph*)}, ref={\rm\alph*}]
\item\label{p:other-a} If\, $\X$ is of\, $K3^{[n]}$ type, then, see \cite[Proposition 2.9]{OS1}, $d:=|G|\leq 66$ and
$$
d\in\{2,3,4,5, 6, \ldots,27, 28,30, 32,33,34,36,38,40,42,44,46,48,50,54,60,66 \} . 
$$ 
For any such Enriques manifold of index $d\in\{13,17,19,23,46\}$, 
we have that $G$ acts as the identity on $\Pic(\X)$.
\item\label{p:other-b} If\, $\X$ is of\, $\Km_n$ type, then, see \cite[Proposition 2.9]{OS1},  $d:=|G|\leq 18$ and
$$d\in\{2,3,4,5,6,7,8,9,10,12,14,18\}. $$
For any such Enriques manifold of index $d\in\{5,7,9,14,18\}$, we have that $G$ acts as the identity on $\Pic(\X)$.
\end{enumerate}
\end{prop}

\begin{proof}
  \eqref{p:other-a}~ For order $d\in\{13,17,19,23\}$ observe that $g^*$ cannot act on $\Pic(\X)_{\IC}$ with primitive roots of unity since the rank is too small: we discuss here the case $d=13$ in details, the other cases being similar.  Since the eigenvalues of the automorphism $g^*$ on $T_{\tilde X}\otimes\IC$ are the primitive roots of unity, we have that $\rk (T_{\tilde X})\geq 12$ (recall that if a primitive root of unity is an eigenvalue for the action of $g^*$ on $T_{\tilde X}\otimes\mathbb C$, then all the others primitive roots are eigenvalues too, with the same multiplicity). This implies that $\rk\Pic(\X)\leq 11$. Now if $g^*$ had an eigenvalue which is a primitive root of unity that is also an eigenvalue  for its action on $\Pic(\X)_{\mathbb C}$, then $\rk\Pic(\X)\geq 13$ (since $\Pic(\X)$ also contains an invariant ample class), but this is in contradiction with the previous inequality.
 
For order $46$ we have $\varphi(46)=22$, and by \cite[Proposition 6(ii)]{Beau-katata} we have that the Picard number of $\X$ is~$1$.  Now we know that $\X$ always contains an ample invariant class, so  the action of $G$ is trivial on $\Pic(\X)$.
  
\eqref{p:other-b}~ For order $d\in\{5,7\}$ observe that $g$ cannot act on $\Pic(\X)_{\IC}$ with some primitive roots of unity since the rank is too small; the argument is the same as in part~\eqref{p:other-a}. For orders $9$, $14$ and $18$ we have $\varphi(9)=6$, $\varphi(14)=6$ and $\varphi(18)=6$, and again by \cite[Proposition 6(ii)]{Beau-katata} we remark that the rank of the Picard group is forced to be equal to~$1$, and we are done as there is always an ample invariant class.
\end{proof} 

\begin{rmk}
In the case where $d=3$ and $\X$ is of $K3^{[2]}$-type, notice that by  \cite[Table 1]{BCS}  a non-symplectic automorphism of order $3$ on a $K3^{[2]}$-type manifold never has  empty fixed locus. So one cannot use  $K3^{[2]}$-type manifolds to construct Enriques manifolds of index $3$. Nevertheless,  as soon as we consider $K3^{[n]}$-type manifolds, with bigger $n$ so that $n+1$ is divisible by $3$,  this may \textit{a priori} be possible.
\end{rmk}

\subsection*{The O'Grady examples}
The two examples of O'Grady $\OG{10}$ and $\OG{6}$ are $10$- and $6$-dimensional, so by Proposition~\ref{prop:d}, to produce Enriques manifolds, we have to take the quotient by a fixed-point-free automorphism of order $2$, $3$ or $6$, respectively of order $2$ or $4$. Recently Billi, Giovenzana, Giovenzana and Grossi \cite{BGGG} have shown the non-existence of Enriques manifolds from $\OG{10}$ manifolds. The possibility of obtaining Enriques manifolds from $\OG{6}$ manifolds remains open.

We now combine the results presented  previously to prove Theorem~\ref{thm:main}. 

\begin{proof}[Proof of Theorem~\ref{thm:main}]
Items~\eqref{t:main-1},~\eqref{t:main-2} and~\eqref{t:main-3}  correspond, respectively, to Propositions~\ref{prop:gen},~\ref{prop:ind4} and~\ref{prop:other}. The nef and movable cone conjectures then follow from Theorem~\ref{identity}. 
\end{proof}


\providecommand{\doi}[2][]{\href{https://doi.org/#2}{{\nolinkurl{doi:#2}}}}
\newcommand{\etalchar}[1]{$^{#1}$}

\end{document}